\documentclass[11pt]{article}
\usepackage[utf8]{inputenc}
\usepackage{geometry}
\usepackage{amsmath, amssymb, amsthm}
\usepackage{graphicx}
\usepackage{bm}
\usepackage{ascmac}
\usepackage{fullpage}

\newtheorem{lemma}{Lemma}

\usepackage{url}
\usepackage{mathtools}
\usepackage{booktabs}
\usepackage[linesnumbered,ruled]{algorithm2e}
\usepackage{multirow}
\title{Structure-Aware Optimization of Decision Diagrams for Health Guidance via Integer Programming}

\author{Nanako Shimaoka$^{1}$ 
\and
Naoyuki Kamiyama$^{2}$ 
\and
Shinji Hotta$^{1}$ 
\and
Sayuri Kohmura$^{1}$
\and
Yuta Kurume$^{1}$
\and
Hiroko Suzuki$^{1}$
\and
Akihiro Inomata$^{1}$ 
\and
Eigo Segawa$^{1}$ 
}

\date{$^1$Fujitsu Limited\\
\url{{shimaoka.nanako, hotta_s, kohmura.sayuri}@fujitsu.com}\\
\url{{kurume.yuta, hiroko.suzuki, akiino, segawa.eigo}@fujitsu.com}\vspace{1mm}\\
$^2$Institute of Mathematics for Industry, Kyushu University
\\
\url{kamiyama@imi.kyushu-u.ac.jp}
}

\begin{document}

\maketitle

\begin{abstract}
In this paper, we consider 
a structure-aware optimization problem for
decision diagrams used for health guidance.
In particular, we focus on decision diagrams
that decide to whom public sectors suggest
consulting a medical worker. 
Furthermore, these diagrams decide which 
notification method should be used
for each target person. 
In this paper, we formulate 
this problem as an integer program.
Then we evaluate its practical usefulness 
through numerical examples.
\end{abstract}

\section{Introduction}

One of the main purposes 
of Japanese public health care 
is to prevent illnesses from becoming severe and to save on future medical costs.
To this end, public sectors often suggest to a part of citizens that they consult 
a medical worker via health guidance based on the results of health checks.
When they decide to whom they suggest
consulting a medical worker, 
decision-making rules
made based on 
specialized knowledge about public health care
are commonly used.
In this situation, 
decision diagrams 
are frequently used 
because decision making by using decision diagrams is 
highly interpretable and transparent. 

The main topic of this paper is 
the structure-aware optimization of
decision diagrams used for health guidance.
More concretely, we consider 
the problem of 
finding a decision diagram 
minimizing a key performance indicator under a cost condition. 
Particularly, we focus on decision diagrams
decide to whom public sectors suggest
consulting a medical worker and 
how to notify the result for each target person. 
We formulate 
this problem as an integer program, and 
we evaluate the practical usefulness of our 
approach through numerical examples motivated 
by applications to health guidance 
in Japan.  

In this paper, 
we consider the following optimization problem. 
We are given a finite set $I$ of health checkup items, 
a finite set $M$ of notification methods (including ``no suggestion''), and 
a directed acyclic graph $D$ with a single source 
such that each non-sink vertex of $D$ is assigned a subset of health checkup
items
$I$ 
and each sink of $D$ is assigned a notification method in $M$.
In applications, 
the input graph $D$ represents the current decision rule. 
In addition, 
we are given 
a finite set of vectors in $\{0,1\}^I$, which 
represents the set of examinee types. 
The goal is to find a decision diagram 
minimizing a key performance indicator under a cost condition.

Our problem seems to be similar to 
the problem of \emph{optimizing decision trees} (see, e.g., \cite{BertsimasD17}).
In this problem, 
we are given a set of data with classification labels.
Then the goal is 
to find a small decision tree which is consistent with
the data or minimizes the misclassification. 
In the following points, 
our problem is different from the problem of 
optimizing decision trees.
\begin{itemize}
\item
In the problem of optimizing decision trees, 
classification labels are given to 
samples in the input data, and the goal is 
to minimize the misclassification. 
On the other hand, in our problem, 
no classification label is given to 
the input data, and the goal is 
to control the subset of samples assigned 
to each sink. 
\item
In the problem of optimizing decision trees,
we construct an optimal decision tree from scratch. 
On the other hand, we care about the 
difference between the resulting decision diagram 
and the initial one.
\item
Basically, in the problem of optimizing decision trees, 
there does not exist a constraint on the resulting 
decision tree. 
On the other hand, in our model, 
we could consider constraints on the resulting decision diagram
arising from practical applications. 
\end{itemize} 

In \cite{SuzukiKNKSS25}, 
a problem similar to ours was considered, focusing on threshold optimization within a fixed decision diagram structure.
In this problem, 
we are given a decision diagram such that the constraint for each 
internal vertex is represented by an inequality with 
a threshold value.
Then the goal is to optimize the threshold 
value for each non-sink vertex in such a 
way that an objective function is optimized. 
The problem considered in 
\cite{SuzukiKNKSS25} and our problem 
are different 
in the following points.
\begin{itemize}
\item
We consider the difference between the resulting decision diagram 
and the input decision diagram. 
On the other hand, in \cite{SuzukiKNKSS25}, this kind of objective value 
is not considered. 
\item
We optimize the combination of conditions assigned to each non-sink vertex, which can
include an empty set, thereby enabling effective structural modification of the decision diagram.
On the other hand, in 
\cite{SuzukiKNKSS25}, the condition assigned to 
each non-sink vertex is basically fixed. 
\end{itemize}

{\bf Related work.} 
As said above, the most relevant problem to our problem is 
the problem of optimizing decision trees. 
It is known that 
a typical setting of this problem is 
NP-complete~\cite{HyafilR76}. 
Practical algorithms for finding 
a good decision tree have been proposed 
(e.g., CART~\cite{Breiman17}). 
This approach has been applied for the 
problem of determining optimal treatments in public health 
(see, e.g., \cite{AmramDZ22,BertsimasDM19,BertsimasKMN22}). 

In \emph{decision tree learning}, 
the aim is to construct a decision tree 
that can output the same result as an unknown true 
decision tree by making queries to some 
oracles (e.g., an equivalence oracle and 
a membership oracle).   
Several models have been proposed.
For example, 
the Probably Approximately Correct (PAC) learning
model was 
introduced by Valiant~\cite{Valiant84}, and 
the model of exact learning from equivalence queries
was introduced by Angluin~\cite{Angluin87} and Littlestone~\cite{Littlestone87}.
Our problem and decision tree learning are different in the point that 
samples are given as 
a part of the input data in our problem. 

{\bf Organization.}
The rest of this paper is organized as follows. 
In Section~\ref{section:problem_formulation}, we formally 
define our problem. 
In Section~\ref{section:integer_program}, we formulate 
our problem as an integer program.
In Section~\ref{section:numerical_example}, 
we show results of numerical examples. 
Section~\ref{section:conclusion} 
concludes this paper. 

\section{Problem Formulation}
\label{section:problem_formulation}

In this section, we formally define 
a structure-aware optimization problem for decision
diagrams.

In our problem, we are given a finite set $I$ of \emph{health checkup items} 
(e.g., $\mbox{HbA1c} \ge 6.5$ and $\mbox{eGFR} < 30$) 
and 
a finite set $M$ of
\emph{notification methods} 
(e.g., mail, e-mail, and telephone).
Let $T$ be a finite set of \emph{examinee types}.
For each examinee type $t \in T$, 
we are given a positive integer $\omega_t$, which 
indicates 
the number of examinees having 
the examinee type $t$. 
For each examinee type $t \in T$, 
we are given a vector $X_t \in \{0,1\}^I$. 
For each health checkup item $i \in I$, 
$X_t(i) = 1$ (resp.\ $X_t(i) = 0$)
means that 
the examinee type $t$ is positive 
(resp.\ negative) 
for $i$. 
For each examinee type $t \in T$, 
we are given a vector $Y_t \in \{0,1\}^M$. 
For each notification method $m \in M$, 
$Y_t(m) = 1$ (resp.\ $Y_t(m) = 0$)
means that the examinee type 
$t$ reacts positively 
(resp.\ negatively) 
for $m$. 
Furthermore, 
we are given a vector $Z \in \{0,1\}^T$. 
For each examinee type $t \in T$, 
$Z(t) = 1$ means that 
if the examinee type $t$ reacts positively for the assigned 
notification method, then 
the result on some pre-specified health checkup 
item (e.g., HbA1c) will be improved. 

In this paper, we are given 
a directed acyclic graph $D = (V,A)$
that contains a single source $r$.\footnote{%
A directed graph is said to be 
\emph{acyclic} if 
it does not contain a directed cycle.
A vertex of a directed acyclic graph 
is called a \emph{source} 
(resp.\ \emph{sink}) 
if no arc enters (resp.\ leaves) 
this vertex.}
In our problem, $D$ represents the structure of 
a decision diagram.
Let $S$ be the set of sinks of $D$. 
Define $U \coloneqq V \setminus S$.  
Furthermore, we assume that, 
for every vertex $u \in U$, 
the following conditions are satisfied. 
(i) 
Exactly two arcs leave $u$. 
(ii)
Exactly one arc leaving $u$ has 
the label $0$, and 
the other arc has the label $1$.  
For each vertex $v \in V \setminus \{r\}$ and 
each integer $\ell \in \{0,1\}$, 
we define $\Gamma_{\ell}(v)$ as the set 
of vertices $u \in U$ such that there exists 
an arc in $A$ from $u$ to $v$ whose label is $\ell$. 

Our goal is to optimize an objective function 
(i.e., a key performance indicator) 
by assigning a subset of health checkup items in $I$
to each vertex in $U$ and a notification method in $M$ to 
each sink in $S$. 
A mapping $\phi \colon V \to 2^{I \cup M}$ 
is called an \emph{assignment} on $D$ if 
(i) $\phi(u) \subseteq I$ for every 
vertex $u \in U$, and 
(ii) $\phi(s) \subseteq M$ and 
$|\phi(s)| = 1$ for every sink $s \in S$.
For each assignment $\phi$ on $D$ and 
each sink $s \in S$, 
we do not distinguish between 
$\phi(s)$ and the unique element in $\phi(s)$. 
For each vertex $u \in U$, we are given 
a subset $\mathbb{C}_u \subseteq 2^I$, 
which represents the set of 
health checkup 
items that can be assigned to $u$. 
In applications of our model, a policy maker can control 
health checkup items assigned to the vertex $u$ by appropriately 
setting $\mathbb{C}_u$. 
An assignment $\phi$ on $D$ is said to be 
\emph{feasible} if 
$\phi(u) \in \mathbb{C}_u$
for every vertex $u \in U$. 
Finally, we are given 
an \emph{initial assignment} $\phi_{\rm in}$ on $D$, which
represents the current decision diagram. 

For each examinee type $t \in T$ and 
each subset $c \subseteq I$, 
we define $c(t)$ by 
\begin{equation*}
c(t) \coloneqq 
\begin{cases}
1 & \mbox{$X_t(i) = 1$ for some health checkup item $i \in c$}\\
0 & \mbox{otherwise}.
\end{cases}
\end{equation*}
Furthermore, for each vertex $u \in U$, 
each examinee type $t \in T$, and 
each integer $\ell \in \{0,1\}$, 
we define $\mathbb{C}_u(t,\ell)$ as 
the set of elements $c \in \mathbb{C}_u$ 
such that $c(t) = \ell$.

Assume that we are given an assignment 
$\phi$ and an examinee type $t \in T$. 
Define $\phi(t)$, which is the notification method 
reached when we apply $\phi$ for $t$, as follows.
First, we compute $(\phi(r))(t)$, and 
then we move to the vertex $v$ such that the edge 
between $r$ and $v$ has the label $(\phi(r))(t)$. 
If $v \in S$, 
then we define $\phi(t) \coloneqq \phi(v)$. 
Otherwise, we repeat the above operation for 
$v$.

For each notification method $m \in M$, 
we are given a non-negative integer ${\sf cost}(m)$, 
which is the cost of the notification method $m$
for one examinee. 
Furthermore, we are given a positive 
integer $B$, 
which represent the upper bound of the total cost.
We define the \emph{cost} of a feasible assignment $\phi$
by 
\begin{equation*} 
{\sf cost}(\phi) \coloneqq \sum_{m \in M}
{\sf cost}(m)
\sum_{t \in T \colon \phi(t) = m}
\omega_t.
\end{equation*}

From here, we consider an objective function. 
Let $\phi$ be a feasible assignment.
Define 
\begin{equation*}
{\sf obj}_1(\phi) \coloneqq 
|\{v \in V \mid \phi(v) = \phi_{\rm in}(v)\}|.
\end{equation*}
That is, ${\sf obj}_1(\phi)$ indicates
the similarity between the input assignment and 
the resulting assignment. 
Furthermore, we  define 
\begin{equation*}
{\sf obj}_2(\phi) \coloneqq 
\sum_{t \in T \colon Y_t(\phi(t))=1}
\omega_t.
\end{equation*}
That is, 
${\sf obj}_2(\phi)$ indicates 
the number of 
examinees that react positively 
for an assigned notification method in $\phi$.
Finally, we define 
\begin{equation*}
{\sf obj}_3(\phi) \coloneqq 
\sum_{t \in T \colon Y_t(\phi(t))=1, Z(t) = 1} \omega_t.
\end{equation*}
That is, 
${\sf obj}_3(\phi)$ indicates 
the number of examinees such that they react 
positively 
for assigned notification methods in $\phi$
and their results on the specified health checkup 
item will be improved. 
We are given positive integers 
$\Theta_1$, 
$\Theta_2$, and 
$\Theta_3$, which 
represent target values 
for 
${\sf obj}_1(\phi)$, 
${\sf obj}_2(\phi)$, and 
${\sf obj}_3(\phi)$, 
respectively. 
In our numerical examples, we set $\Theta_1 \coloneqq |V|$. 
Then we consider the following 
settings. 
\begin{description}
\item[Setting~1.]
The goal is to find an assignment $\phi$ on $D$ maximizing
$\sum_{i=1}^3
({\sf obj}_i(\phi) / \Theta_i)$
under the condition that ${\sf cost}(\phi) \le B$.
\item[Setting~2.]
The goal is to find an assignment $\phi$ on $D$ minimizing 
${\sf cost}(\phi)$
under the condition that 
${\sf obj}_1(\phi) \ge \Theta_1/2$ and 
${\sf obj}_i(\phi) \ge \Theta_i$
for every integer $i \in \{2,3\}$.
\item[Setting~3.]
The goal is to find an assignment $\phi$ on $D$ maximizing
${\sf obj}_1(\phi)$
under the condition that 
${\sf cost}(\phi) \le B$ and 
${\sf obj}_i(\phi) \ge \Theta_i$
for every integer $i \in \{2,3\}$.
\end{description} 
In our numerical examples, 
we investigate the differences among these three settings 
from the viewpoint of the 
computational time and the quality of the solution. 

\section{Integer Program Formulation}
\label{section:integer_program}

In this section, we formulate our problem as an integer program. 

The following variable $p$ (resp.\ $q$) indicates 
which subset of $I$ (resp.\ notification method in $M$) 
is assigned to each vertex in $U$ (resp.\ $S$). 
\begin{equation} \label{eq:p_q}
p_{u,c} \in \{0,1\} 
\ \ \ \mbox{($\forall u \in U$, $\forall c \in \mathbb{C}_u$)}, \ \ \ \ \ 
q_{s,m} \in \{0,1\} 
\ \ \ \mbox{($\forall s \in S$, $\forall m \in M$)}.
\end{equation}
We say that \emph{an assignment $\phi$ corresponds to 
$p,q$} if the following conditions are satisfied. 
\begin{equation*}
\mbox{
$p_{u,c} = 1$ if and only if 
$\phi(u) = c$, \ \ \ and \ \ \  
$q_{s,m} = 1$ if and only if 
$\phi(s) = m$}. 
\end{equation*}
We consider the following constraints.
\begin{equation} \label{eq:constraint_p_q}
\sum_{c \in \mathbb{C}_u}p_{u,c} = 1
\ \ \ \mbox{($\forall u \in U$)}, \ \ \ \ \ 
\sum_{m \in M}q_{s,m} = 1
\ \ \ \mbox{($\forall s \in S$)}. 
\end{equation}
Then 
it is easy to see that 
there exists a feasible assignment $\phi$ on $D$
corresponding to 
$p,q$ if and only if 
\eqref{eq:constraint_p_q} is satisfied. 

Next, we consider the following variables. 
\begin{equation} \label{eq:alpha_beta}
\alpha_{t,v} \in \{0,1\} 
\ \ \ \mbox{($\forall t \in T$, $\forall v \in V$)}, \ \ \ \ \ 
\beta_{t,u,\ell} \in \{0,1\} 
\ \ \ \mbox{($\forall t \in T$, $\forall u \in U$, $\forall \ell \in \{0,1\}$)}. 
\end{equation}
We aim to guarantee that 
(i) $\alpha_{t,v}=1$ if and only if 
we pass through $v$ when we apply $\phi$ for $t$, and 
(ii) 
$\beta_{t,u,\ell} = 1$ if and only if 
$\phi(u) \in \mathbb{C}_u(t,\ell)$ and 
we pass through $u$ when we apply $\phi$ for $t$.
To this end, we consider the following constraints. 
\begin{equation} \label{eq:constraint_alpha} 
\begin{split}
& \alpha_{t,r} = 1
\ \ \ \mbox{($\forall t \in T$)}.\\
& \alpha_{t,v} \le \sum_{\ell \in \{0,1\}}\sum_{w \in \Gamma_{\ell}(v)}
\beta_{t,w,\ell} 
\ \ \ \mbox{($\forall t \in T$, $\forall v \in V \setminus \{r\}$)}.\\
& \alpha_{t,v} \ge \beta_{t,w,\ell} 
\ \ \ \mbox{($\forall t \in T$, $\forall v \in V \setminus \{r\}$,
$\forall \ell \in \{0,1\}$,
$\forall w \in \Gamma_{\ell}(v)$)}.
\end{split} 
\end{equation}

\begin{lemma} \label{lemma_1:alpha_beta}
Assume that we are given a feasible assignment 
$\phi$ corresponding to $p,q$ satisfying \eqref{eq:constraint_p_q}, 
and assume that $\alpha,\beta$ satisfy 
\eqref{eq:constraint_alpha}.
Let $t,v$ be 
an examinee type in $T$ and
a vertex in $V \setminus \{r\}$, respectively. 
Assume that, for 
every integer 
$\ell \in \{0,1\}$ and
every vertex $w \in \Gamma_\ell(v)$, 
$\beta_{t,w,\ell} = 1$ if and only if 
$\phi(w) \in \mathbb{C}_w(t,\ell)$ and 
we pass through $w$ when we apply $\phi$ for $t$.
Then $\alpha_{t,v} = 1$ 
if and only if 
we pass through $v$ when we apply $\phi$ for $t$. 
\end{lemma}
\begin{proof}
Assume that 
$\alpha_{t,v} = 1$. 
The second constraint of \eqref{eq:constraint_alpha} implies that 
there exist an integer $\ell \in \{0,1\}$ and a vertex $w \in \Gamma_{\ell}(v)$
such that $\beta_{t,w,\ell} = 1$. 
This and the assumption of this lemma imply that 
$\phi(w) \in \mathbb{C}_w(t,\ell)$ and 
we pass through $w$ when we apply $\phi$ for $t$.
This implies that 
we pass through $v$ when we apply $\phi$ for $t$. 

Conversely, we assume that 
we pass through $v$ when we apply $\phi$ for $t$. 
Then 
there exist an integer $\ell \in \{0,1\}$ and a vertex $w \in \Gamma_{\ell}(v)$
such that 
$\phi(w) \in \mathbb{C}_w(t,\ell)$ and 
we pass through $w$ when we apply $\phi$ for $t$.
Thus, 
the assumption of this lemma implies that 
$\beta_{t,w,\ell} = 1$. 
This and the third constraint of \eqref{eq:constraint_alpha} implies that
$\alpha_{t,v}=1$. 
\end{proof} 

Furthermore, we consider the following constraints. 
\begin{equation} \label{eq:constraint_beta}
\begin{split}
& \beta_{t,u,\ell} \le \alpha_{t,u}
\ \ \ \mbox{($\forall t \in T$,
$\forall u \in U$,
$\forall \ell \in \{0,1\}$)}.\\
& \beta_{t,u,\ell} \le \sum_{c \in \mathbb{C}_u(t,\ell)}p_{u,c}
\ \ \ \mbox{($\forall t \in T$, $\forall u \in U$, $\forall \ell \in \{0,1\}$)}.\\
& \beta_{t,u,\ell} \ge \alpha_{t,u} + \sum_{c \in \mathbb{C}_u(t,\ell)}p_{u,c} -1
\ \ \ \mbox{($\forall t \in T$, $\forall u \in U$, $\forall \ell \in \{0,1\}$)}. 
\end{split} 
\end{equation}

\begin{lemma} \label{lemma_2:alpha_beta}
Assume that we are given a feasible assignment 
$\phi$ corresponding to $p,q$ satisfying \eqref{eq:constraint_p_q},
and assume 
that $\alpha,\beta$ satisfy 
\eqref{eq:constraint_beta}. 
Let $t,u$ be 
an examinee type in $T$ and
a vertex in $U$, respectively. 
Assume that 
$\alpha_{t,u} = 1$ 
if and only if 
we pass through $u$ when we apply $\phi$ for $t$.
Then
for every integer $\ell \in \{0,1\}$, 
$\beta_{t,u,\ell} = 1$ if and only if 
$\phi(u) \in \mathbb{C}_u(t,\ell)$ and 
we pass through $u$ when we apply $\phi$ for $t$.
\end{lemma}
\begin{proof}
Let $\ell$ be an integer in $\{0,1\}$.

Assume that 
$\beta_{t,u,\ell} = 1$. 
Then the first condition of \eqref{eq:constraint_beta} implies that 
$\alpha_{t,u} = 1$. 
This and the assumption of this lemma imply that 
we pass through $u$ when we apply $\phi$ for $t$. 
Furthermore, the second condition of \eqref{eq:constraint_beta} implies that 
there exists an element $c \in \mathbb{C}_u(t,\ell)$
such that $p_{u,c} = 1$. 
Thus, since $\phi$ corresponds to $p,q$, 
we have 
$\phi(u) \in \mathbb{C}_u(t,\ell)$. 

Conversely, we assume that 
$\phi(u) \in \mathbb{C}_u(t,\ell)$ and 
we pass through $u$ when we apply $\phi$ for $t$.
Then the assumption of this lemma implies that 
$\alpha_{t,u}=1$. 
In addition, since 
$\phi(u) \in \mathbb{C}_u(t,\ell)$ and 
$\phi$ corresponds to $p$, 
$\sum_{c \in \mathbb{C}_u(t,\ell)}p_{u,c}=1$. 
The third condition of \eqref{eq:constraint_beta} implies that 
$\beta_{t,u,\ell}=1$. 
\end{proof}

\begin{lemma} \label{lemma_3:alpha_beta}
Assume that we are given a feasible assignment 
$\phi$ corresponding to $p,q$ satisfying \eqref{eq:constraint_p_q},
and 
assume that $\alpha,\beta$ satisfy 
\eqref{eq:constraint_alpha}, 
\eqref{eq:constraint_beta}. 
Then for 
every examinee type $t \in T$ and
every vertex $v \in V$, 
$\alpha_{t,v} = 1$ 
if and only if 
we pass through $v$ when we apply $\phi$ for $t$. 
\end{lemma}
\begin{proof}
Let $t,v$ 
be 
an examinee type in $T$ and
a vertex in $V$, respectively. 
We prove this lemma by induction on the maximum number of arcs of 
a path in $D$ from $r$ to $v$. 

First, we assume that $v = r$. 
The first constraint of 
\eqref{eq:constraint_alpha} implies that 
$\alpha_{t,v} = 1$ always hold. 
Furthermore, 
we clearly pass through $r$ when we apply $\phi$ for $t$. 

Next, we assume that 
$v \neq r$.
Assume that, 
for 
every integer 
$\ell \in \{0,1\}$ and 
every vertex $w \in \Gamma_{\ell}(v)$, 
$\alpha_{t,w} = 1$ if and only if 
we pass through $w$ when we apply $\phi$ for $t$.
Then 
Lemma~\ref{lemma_2:alpha_beta}
implies that, 
for  
every integer 
$\ell \in \{0,1\}$ and 
every vertex $w \in \Gamma_{\ell}(v)$, 
$\beta_{t,w,\ell} = 1$ if and only if 
$\phi(w) \in \mathbb{C}_w(t,\ell)$ and 
we pass through $w$ when we apply $\phi$ for $t$.
Thus, 
Lemma~\ref{lemma_1:alpha_beta} 
implies that 
$\alpha_{t,v} = 1$ 
if and only if 
we pass through $v$ when we apply $\phi$ for $t$.
This completes the proof. 
\end{proof} 

Third, we consider the following variables. 
\begin{equation} \label{eq:gamma}
\gamma_{t,s,m} \in \{0,1\} 
\ \ \ \mbox{($\forall t \in T$, $\forall s \in S$, $\forall m \in M$)}.
\end{equation}
We aim to guarantee that 
$\gamma_{t,s,m}=1$ if and only if
$m$ is assigned to $s$
and we reach $s$ 
when we apply $\phi$ for $t$. 
To this end, we consider the following conditions. 
\begin{equation} \label{eq:constraint_gamma}
\begin{split}
& 
\gamma_{t,s,m} \le q_{s,m}
\ \ \ \mbox{($\forall t \in T$, $\forall s \in S$, $\forall m \in M$)}.
\\ 
& 
\gamma_{t,s,m} \le \alpha_{t,s}
\ \ \ \mbox{($\forall t \in T$, $\forall s \in S$, $\forall m \in M$)}.
\\ 
& 
\gamma_{t,s,m}  \ge q_{s,m} + \alpha_{t,s} - 1
\ \ \ \mbox{($\forall t \in T$, $\forall s \in S$, $\forall m \in M$)}. 
\end{split}
\end{equation}

\begin{lemma} \label{lemma:gamma}
Assume that we are given a feasible assignment 
$\phi$ corresponding to $p,q$ satisfying \eqref{eq:constraint_p_q},
and assume that $\alpha,\beta,\gamma$ satisfy 
\eqref{eq:constraint_alpha}, 
\eqref{eq:constraint_beta},
\eqref{eq:constraint_gamma}.
Then for 
every examinee type $t \in T$,
every sink $s \in S$, and 
every notification method $m \in M$, 
$\gamma_{t,s,m}=1$ if and only if 
$m$ is assigned to $s$
and we reach $s$ 
when we apply $\phi$ for $t$. 
\end{lemma}
\begin{proof}
Let $t$, $s$, and $m$ be 
an examinee type in $T$,
a sink in $S$, and 
a notification method in $M$, 
respectively. 

If $\gamma_{t,s,m}=1$, then the first and second constraints of \eqref{eq:constraint_gamma}
imply that 
$q_{s,m}=1$ and $\alpha_{t,s}=1$.
Thus, Lemma~\ref{lemma_3:alpha_beta} 
implies that 
$m$ is assigned to $s$
and 
we reach $s$ 
when we apply $\phi$ for $t$. 

Conversely, 
we assume that 
$m$ is assigned to $s$
and 
we reach $s$ 
when we apply $\phi$ for $t$. 
Then 
$q_{s,m}=1$ and $\alpha_{t,s}=1$.
Thus, 
the third constraint of \eqref{eq:constraint_gamma}
implies that 
$\gamma_{t,s,m}=1$. 
\end{proof} 

Finally, we consider the following variables. 
\begin{equation} \label{eq:z}
z_{t,m} \in \{0,1\} 
\ \ \ \mbox{($\forall t \in T$, $\forall m \in M$)}.
\end{equation}
Then we aim to 
guarantee that 
$z_{t,m} = 1$ if and only if  
$\phi(t) = m$.
To this end, we consider the following conditions. 
\begin{equation} \label{eq:constraint_z}
\begin{split}
& z_{t,m} \le \sum_{s \in S} \gamma_{t,s,m}
\ \ \ \mbox{($\forall t \in T$, $\forall m \in M$)} 
\\ 
& 
z_{t,m} \ge 
\gamma_{t,s,m} 
\ \ \ \mbox{($\forall t \in T$, $\forall s \in S$, $\forall m \in M$)}.
\end{split} 
\end{equation}

\begin{lemma}
Assume that we are given a feasible assignment 
$\phi$ corresponding to $p,q$ satisfying \eqref{eq:constraint_p_q},
and assume that $\alpha,\beta,\gamma$ satisfy 
\eqref{eq:constraint_alpha}, 
\eqref{eq:constraint_beta},
\eqref{eq:constraint_gamma},
\eqref{eq:constraint_z}.
Then
for 
every examinee type $t \in T$ and 
every notification method $m \in M$, 
$z_{t,m} = 1$ if and only if 
$\phi(t)=m$.
\end{lemma}
\begin{proof}
Let $t$ and $m$ be 
an examinee type in $T$ and 
a notification method in $M$, 
respectively. 

Assume that $z_{t,m}=1$. 
Then the first constraint
of \eqref{eq:constraint_z} 
implies that
there exists a sink $s \in S$ such that 
$\gamma_{t,s,m} = 1$. 
Thus, Lemma~\ref{lemma:gamma} 
implies that 
$m$ is assigned to $s$
and
we reach $s$ 
when we apply $\phi$ for $t$. 
This implies that 
$\phi(t)=m$.

Conversely, we assume that 
$\phi(t)=m$.
Then 
there exists a sink $s \in S$ such that 
$m$ is assigned to $s$
and
we reach $s$ 
when we apply $\phi$ for $t$. 
This and  
Lemma~\ref{lemma:gamma} imply that 
$\gamma_{t,s,m} =1$. 
Thus, the second constraint 
of \eqref{eq:constraint_z} 
imply that $z_{t,m} = 1$.
\end{proof} 

Next, we consider the objective functions. 
It is not difficult to see that 
if the variables \eqref{eq:p_q}, \eqref{eq:alpha_beta}, \eqref{eq:gamma}, 
\eqref{eq:z} satisfy the constraints \eqref{eq:constraint_p_q}, \eqref{eq:constraint_alpha}, \eqref{eq:constraint_beta}, 
\eqref{eq:constraint_gamma}, \eqref{eq:constraint_z}, 
then  
\begin{equation*} 
{\sf cost}(\phi) = 
\sum_{m \in M}
{\sf cost}(m)  
\sum_{t \in T}
\omega_t \cdot z_{t,m}.
\end{equation*}
Furthermore, we have 
\begin{equation*} 
\begin{split}
{\sf obj}_1(\phi) & = 
\sum_{u \in U}p_{u,\phi_{\rm in}(u)}
+ 
\sum_{s \in S}q_{s,\phi_{\rm in}(s)}
\\
{\sf obj}_2(\phi) & = 
\sum_{t \in T}
\sum_{m \in M}
Y_t(m) \cdot \omega_t \cdot z_{t,m}\\
{\sf obj}_3(\phi) & = 
\sum_{t \in T}
\sum_{m \in M}
Y_t(m) \cdot Z(t) \cdot \omega_t \cdot z_{t,m}.
\end{split}
\end{equation*}

We are now ready to formulate our problem as an integer program. 
The variables are \eqref{eq:p_q}, \eqref{eq:alpha_beta}, \eqref{eq:gamma}, and 
\eqref{eq:z}. 
Then the setting~1 can be formulated as follows. 
\begin{equation*}
\begin{array}{cl}
\mbox{Maximize} & \displaystyle{\sum_{i=1}^3 \dfrac{{\sf obj}_i(\phi)}{\Theta_i}} \vspace{2mm} \\ 
\mbox{subject to} & {\sf cost}(\phi) \le B \vspace{1mm} \\ 
& \mbox{\eqref{eq:constraint_p_q}, \eqref{eq:constraint_alpha}, \eqref{eq:constraint_beta}, 
\eqref{eq:constraint_gamma}, and \eqref{eq:constraint_z}}.
\end{array}
\end{equation*}
Furthermore, the setting~2 can be formulated as follows. 
\begin{equation*}
\begin{array}{cl}
\mbox{Minimize} & {\sf cost}(\phi) \vspace{2mm} \\ 
\mbox{subject to} & {\sf obj}_1(\phi) \ge \Theta_1/2 \\
& {\sf obj}_i(\phi) \ge \Theta_i 
\ \ \mbox{($\forall i \in \{2,3\}$)} \vspace{1mm} \\ 
& \mbox{\eqref{eq:constraint_p_q}, \eqref{eq:constraint_alpha}, \eqref{eq:constraint_beta}, 
\eqref{eq:constraint_gamma}, and \eqref{eq:constraint_z}}. 
\end{array}
\end{equation*}
Finally, the setting~3 can be formulated as follows. 
\begin{equation*}
\begin{array}{cl}
\mbox{Maximize} & {\sf obj}_1(\phi) \vspace{2mm} \\ 
\mbox{subject to} & {\sf cost}(\phi) \le B \vspace{1mm} \\ 
& {\sf obj}_i(\phi) \ge \Theta_i \ \ \mbox{($\forall i \in \{2,3\}$)} \vspace{1mm} \\ 
& \mbox{\eqref{eq:constraint_p_q}, \eqref{eq:constraint_alpha}, \eqref{eq:constraint_beta}, 
\eqref{eq:constraint_gamma}, and \eqref{eq:constraint_z}}.
\end{array}
\end{equation*}

\section{Numerical Examples}
\label{section:numerical_example}

In the following numerical examples, 
we used Gurobi Optimizer (version 12.0.1)~\cite{gurobi} to 
solve the integer programs, and 
the program was implemented in Python 3.13.1 and 
run on a MacBook Pro with macOS Sequoia 15.7.3, 
an Apple M4 Pro, and 48 GB RAM.

\subsection{Instances} 

In our experiments, we define 
$I \coloneqq \{0,1,\dots,48\}$ and
$M \coloneqq \{0,1,2,3\}$. 
The notification methods $0, 1, 2, 3$ 
correspond to 
no suggestion, by mail, by telephone, and 
by mail and telephone, respectively.
Define the cost of each notification method 
by
\begin{equation*}
{\sf cost}(0) \coloneqq 0, \ \ 
{\sf cost}(1) \coloneqq 200, \ \ 
{\sf cost}(2) \coloneqq 500, \ \ 
{\sf cost}(3) \coloneqq 700.  
\end{equation*}
In the following examples, 
for every examinee type $t \in T$, 
we define $Y_t(0) \coloneqq 0$. 
For the health checkup items in $I$ and 
how to generate our data, 
see Appendix~\ref{appendix:example}.

We consider the following three 
instances, which are called Instances~1, 2, and 3. 
The diagrams of the instances are illustrated in Figure~\ref{fig:example}. 
The real lines are the arcs having the label $1$, and 
the dashed lines are the arcs having the label $0$.
Furthermore, the input data for each instance is described in 
Table~\ref{tab1}.
For the initial assignments, see Section~\ref{section:result}.

\begin{figure}[ht]
\begin{minipage}{0.28\hsize}
\begin{center}
\includegraphics[width=3cm]{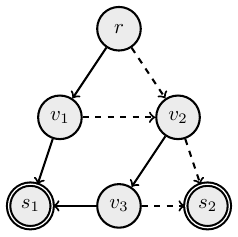}
\par(a)
\end{center}
\end{minipage}
\begin{minipage}{0.35\hsize}
\begin{center}
\includegraphics[width=5cm]{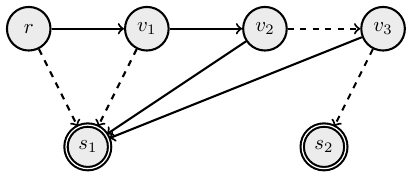}
\par(b)
\end{center}
\end{minipage}
\begin{minipage}{0.35\hsize}
\begin{center}
\includegraphics[width=5cm]{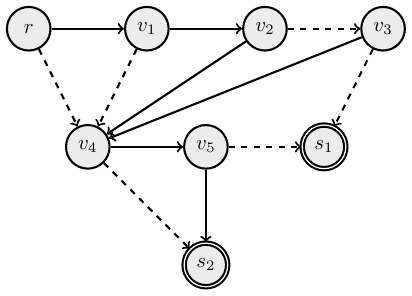}
\par(c)
\end{center}
\end{minipage}
\caption{(a) Instance~1.
(b) Instance~2.
(c) Instance~3.}
\label{fig:example}
\end{figure}

\begin{table}[h]
\caption{The input data for the instances.}
\vspace{1mm}
\label{tab1}
\centering
\begin{tabular}{ccccccc}
\toprule
& $D$ & $|T|$ & $B$ & $\Theta_1$ & $\Theta_2$ & $\Theta_3$ \\
\cmidrule(lr){2-7}
{\bf Instance~1} & Fig.~\ref{fig:example}(a) & 524 & 35000 & 6 & 15 & 9 \\
{\bf Instance~2} & Fig.~\ref{fig:example}(b) & 22450 & 373333 & 6 & 160 & 96 \\
{\bf Instance~3} & Fig.~\ref{fig:example}(c) & 4742 & 483000 & 8 & 207 & 124 \\
\bottomrule
\end{tabular}
\end{table}

For each vertex $u \in U$,
$\mathbb{C}_u$ is defined as follows.
For each vertex $u \in U$, 
we define $\mathbb{D}_u \subseteq 2^I$
as the family of subsets $c \subseteq I$ 
satisfying one of the following 
conditions.
\begin{description}
\item[(C1)]
$c \subseteq \phi_{\rm in}(u)$ and $|\phi_{\rm in}(u) \setminus c| = 1$. 
\item[(C2)]
$\phi_{\rm in}(u) \subseteq c$ and $|c \setminus \phi_{\rm in}(u)| = 1$. 
\item[(C3)]
$|\phi_{\rm in}(u) \setminus c| = 1$ and 
$|c \setminus \phi_{\rm in}(u)| = 1$. 
\end{description}
That is, we try to find an assignment that is 
not drastically different from the initial assignment. 
This constraint reflects 
the thought that, in applications, 
dramatic changes to the rule confuse 
operation. 
Define ${\cal F}^- \subseteq 2^I$ by 
\begin{equation*}
\begin{split}
F^- & \coloneqq 
\{
\{1,2\},
\{3,4\},
\{5,6,7,8,9,10,11\},
\{12,13,14\},
\{15,16,17\},\\
& 
\ \ \ \ \ \ \ \ 
\{18,19,20,21,22,23,24\},
\{25,26,27,28,29,30,31,32,33,34,35\}
\}.
\end{split}
\end{equation*}
Then for each vertex $u \in U$, 
we define $\mathbb{E}_u$ as the set of elements $c \in \mathbb{D}_u$ such that 
$|c \cap F| \le 1$ for every 
element $F \in {\cal F}^-$.
That is, 
the health checkup items are categorized, and 
each vertex can contain at most one item from each 
category.

We define $\mathbb{C}_u$ for each vertex $u \in U$ 
in each instance as follows. 
\begin{itemize}
\item
In Instance~1, we define as follows.
\begin{itemize}
\item
$\mathbb{C}_{r} \coloneqq 
\{c \in \mathbb{E}_{r} \mid c \subseteq \{0\}\}$. 
\item
$\mathbb{C}_{v_1} \coloneqq
\{c \in \mathbb{E}_{v_1} \mid c \subseteq \{1,2,3,4,5,6,7,8,9,10,11,12,13,14,15,16,17\}\}$. 
\item
$\mathbb{C}_{v_2} \coloneqq 
\{c \in \mathbb{E}_{v_2} \mid c \subseteq \{38,39,42,44,46,47,48\}\}$. 
\item
$\mathbb{C}_{v_3} \coloneqq 
\{c \in \mathbb{E}_{v_3} \mid c \subseteq \{36,40,43\}\}$. 
\end{itemize}
\item
In Instance~2, we define as follows.
\begin{itemize}
\item
$\mathbb{C}_{r} \coloneqq 
\{c \in \mathbb{E}_{r} \mid c \subseteq \{0\}\}$. 
\item
$\mathbb{C}_{v_1} \coloneqq 
\{c \in \mathbb{E}_{v_1} \mid c \subseteq \{1,2,3,4,5,6,7,8,9,10,11,12,13,14,15,16,17\}\}$. 
\item
$\mathbb{C}_{v_2} \coloneqq 
\{c \in \mathbb{E}_{v_2} \mid c \subseteq \{18,19,20,21,22,23,24,25,26,27,28,29,30,31,32,33,34,35\}\}$. 
\item
$\mathbb{C}_{v_3} \coloneqq 
\{c \in \mathbb{E}_{v_3} \mid c \subseteq \{36,37,38,39,40,41,42,43,44,45\}\}$. 
\end{itemize}
\item
In Instance~3, we define as follows.
\begin{itemize}
\item
$\mathbb{C}_{r} \coloneqq 
\{c \in \mathbb{E}_{r} \mid c \subseteq \{0\}\}$. 
\item
$\mathbb{C}_{v_1} \coloneqq 
\{c \in \mathbb{E}_{v_1} \mid c \subseteq \{1,2,3,4,5,6,7,8,9,10,11,12,13,14,15,16,17\}\}$. 
\item
$\mathbb{C}_{v_2} \coloneqq 
\{c \in \mathbb{E}_{v_2} \mid c \subseteq \{36,37,38,39,40,41,42,43,44,45\}\}$. 
\item
$\mathbb{C}_{v_3} \coloneqq 
\{c \in \mathbb{E}_{v_3} \mid c \subseteq \{36,37,38,39,40,41,42,43,44,45\}\}$. 
\item
$\mathbb{C}_{v_4} \coloneqq 
\{c \in \mathbb{E}_{v_4} \mid c \subseteq \{38,39,42,44,46,47,48\}\}$. 
\item
$\mathbb{C}_{v_5} \coloneqq 
\{c \in \mathbb{E}_{v_5} \mid c \subseteq \{36,40,43\}\}$. 
\end{itemize}
\end{itemize}
This means that, in each instance, 
each vertex has its own role, and 
can contain only the items related to this role. 

\subsection{Results} 
\label{section:result}

Here we give the results of our experiments. 

The results for Instance~1 are described in 
Tables~\ref{tab_instance_1_1} and \ref{tab_instance_1_2},  
the results for Instance~2 are described in 
Tables~\ref{tab_instance_2_1} and \ref{tab_instance_2_2}, and 
the results for Instance~3 are described in 
Tables~\ref{tab_instance_3_1} and \ref{tab_instance_3_2}. 

\begin{table}[h]
\caption{The result for Instance~1.}
\vspace{1mm}
\label{tab_instance_1_1}
\centering
\begin{tabular}{cccccc}
\toprule
& Time (sec.) & ${\sf cost}(\phi)$ & ${\sf obj}_1(\phi)$ & ${\sf obj}_2(\phi)$ & ${\sf obj}_3(\phi)$ \\
\cmidrule(lr){2-6}
{\bf Input} & $-$ & 15600 & $-$ & 18 & 8 \\
{\bf Setting~1} & 2.57 & 34800 & 4 & 43 & 27 \\
{\bf Setting~2} & 2.44 & 10400 & 3 & 15 & 10 \\
{\bf Setting~3} & 0.98 & 26000 & 5 & 36 & 15 \\
\bottomrule
\end{tabular}
\end{table}

\begin{table}[h]
\caption{The result for Instance~1.}
\vspace{1mm}
\label{tab_instance_1_2}
\centering
\begin{tabular}{ccccccc}
\toprule
& $\phi(r)$ & $\phi(v_1)$ & $\phi(v_2)$ & $\phi(v_3)$ & $\phi(s_1)$ & $\phi(s_2)$ \\ 
\cmidrule(lr){2-7}
{\bf Input} & $\{0\}$ & $\{1,4,8\}$ & $\{44\}$ & $\{40\}$ & $\{1\}$ & $\{0\}$ \\
{\bf Setting~1} & $\{0\}$ & $\{1,4,8,15\}$ & $\{42\}$ & $\{40\}$ & $\{1\}$ & $\{0\}$ \\
{\bf Setting~2} & $\{0\}$ & $\{1,4,17\}$ & $\{39,44\}$ & $\{36\}$ & $\{1\}$ & $\{0\}$ \\
{\bf Setting~3} & $\{0\}$ & $\{1,4,8\}$ & $\{38,44\}$ & $\{40\}$ & $\{1\}$ & $\{0\}$ \\
\bottomrule
\end{tabular}
\end{table}

\begin{table}[h]
\caption{The result for Instance~2.}
\vspace{1mm}
\label{tab_instance_2_1}
\centering
\begin{tabular}{cccccc}
\toprule
& Time (sec.) & ${\sf cost}(\phi)$ & ${\sf obj}_1(\phi)$ & ${\sf obj}_2(\phi)$ & ${\sf obj}_3(\phi)$ \\
\cmidrule(lr){2-6}
{\bf Input} & $-$ & 88000 & $-$ & 135 & 93 \\
{\bf Setting~1} & 819.30 & 370000 & 4 & 576 & 313 \\
{\bf Setting~2} & 1067.69 & 103600 & 4 & 160 & 107 \\
{\bf Setting~3} & 66.56 & 308000 & 5 & 167 & 103 \\
\bottomrule
\end{tabular}
\end{table}

\begin{table}[h]
\caption{The result for Instance~2.}
\vspace{1mm}
\label{tab_instance_2_2}
\centering
\begin{tabular}{ccccccc}
\toprule
& $\phi(r)$ & $\phi(v_1)$ & $\phi(v_2)$ & $\phi(v_3)$ & $\phi(s_1)$ & $\phi(s_2)$ \\ 
\cmidrule(lr){2-7}
{\bf Input} & $\{0\}$ & $\{8\}$ & $\{23\}$ & $\{45\}$ & $\{0\}$ & $\{1\}$ \\
{\bf Setting~1} & $\{0\}$ & $\{8,16\}$ & $\{25\}$ & $\{45\}$ & $\{0\}$ & $\{1\}$ \\
{\bf Setting~2} & $\{0\}$ & $\{8\}$ & $\{19\}$ & $\{36\}$ & $\{0\}$ & $\{1\}$ \\
{\bf Setting~3} & $\{0\}$ & $\{8\}$ & $\{23\}$ & $\{45\}$ & $\{0\}$ & $\{3\}$ \\
\bottomrule
\end{tabular}
\end{table}

\begin{table}[h]
\caption{The result for Instance~3.}
\vspace{1mm}
\label{tab_instance_3_1}
\centering
\begin{tabular}{cccccc}
\toprule
& Time (sec.) & ${\sf cost}(\phi)$ & ${\sf obj}_1(\phi)$ & ${\sf obj}_2(\phi)$ & ${\sf obj}_3(\phi)$ \\
\cmidrule(lr){2-6}
{\bf Input} & $-$ & 144400 & $-$ & 233 & 120 \\
{\bf Setting~1} & 140.36 & 432600 & 3 & 664 & 342 \\
{\bf Setting~2} & 2133.56 & 137200 & 5 & 228 & 124 \\
{\bf Setting~3} & 30.23 & 361000 & 7 & 271 & 125 \\
\bottomrule
\end{tabular}
\end{table}

\begin{table}[h]
\caption{The result for Instance~3.}
\vspace{1mm}
\label{tab_instance_3_2}
\centering
\begin{tabular}{ccccccccc}
\toprule
& $\phi(r)$ & $\phi(v_1)$ & $\phi(v_2)$ & $\phi(v_3)$ & $\phi(v_4)$ & $\phi(v_5)$ & $\phi(s_1)$ & $\phi(s_2)$ \\ 
\cmidrule(lr){2-9}
{\bf Input} & $\{0\}$ & $\{1,8\}$ & $\{44\}$ & $\{45\}$ & $\{39\}$ & $\{36\}$ & $\{1\}$ & $\{0\}$ \\
{\bf Setting~1} & $\{0\}$ & $\{1,5\}$ & $\{\}$ & $\{\}$ & $\{39,48\}$ & $\{\}$ & $\{1\}$ & $\{0\}$ \\
{\bf Setting~2} & $\{0\}$ & $\{8,12\}$ & $\{37\}$ & $\{41,45\}$ & $\{39\}$ & $\{36\}$ & $\{1\}$ & $\{0\}$ \\
{\bf Setting~3} & $\{0\}$ & $\{1,8\}$ & $\{44\}$ & $\{45\}$ & $\{39\}$ & $\{36\}$ & $\{2\}$ & $\{0\}$ \\
\bottomrule
\end{tabular}
\end{table}

Basically, every setting of each instance 
can be solved with reasonable computational time. 
Furthermore, the solutions can improve the key performance 
indicators than the initial decision diagrams.
Our results show that Setting~3 is the easiest one 
among the three settings.  

\section{Concluding Remarks}
\label{section:conclusion} 

In this paper, we propose an integer programming approach to 
optimizing a decision diagram 
used for 
health guidance.  
We formulate 
this problem as an integer program.
Then we evaluate its practical usefulness 
through numerical experiments.

In the model of this paper, we assume that we know  
the effect of health guidance for each examinee type
in advance. 
It would be interesting to consider a model 
where the effect  
of health guidance for each examinee type 
is uncertain. 

\section*{Acknowledgments}

This research was partially conducted by the Fujitsu Small Research 
Lab ``Division of Fujitsu Mathematical Modeling for Decision Making,'' 
a joint research center between 
Fujitsu Limited and Kyushu University.

\bibliographystyle{plain}
\bibliography{decision_diagram_bib}

\clearpage 

\appendix

\section{Details of Numerical Examples}
\label{appendix:example} 

\subsection{Health checkup items}

\begin{itemize}
\item
$\{0\}$ : Health checkup
\item
$\{1,2\}$ : Fasting Blood Glucose 
\begin{itemize}
\item
Item~1: 
Fasting Blood Glucose $\ge 126$ 
\item
Item~2: 
Fasting Blood Glucose $\ge 130$ 
\end{itemize}
\item
$\{3,4\}$ : Casual Blood Glucose
\begin{itemize}
\item
Item~3: 
Casual Blood Glucose $\ge 126$
\item
Item~4: 
Casual Blood Glucose $\ge 200$ 
\end{itemize}
\item
$\{5,6,7,8,9,10,11\}$ : HbA1c
\begin{itemize}
\item
Item~5: 
HbA1c $\ge 5.6$ 
\item
Item~6: 
HbA1c $\ge 6.0$ 
\item
Item~7: 
HbA1c $\ge 6.2$ 
\item
Item~8: 
HbA1c $\ge 6.5$ 
\item
Item~9: 
HbA1c $\ge 7.0$ 
\item
Item~10: 
HbA1c $\ge 8.0$ 
\item
Item~11: 
$6.0 \le$ HbA1c $< 6.5$ 
\end{itemize}
\item
$\{12,13,14\}$ : Diastolic Blood Pressure 
\begin{itemize}
\item
Item~12: 
Diastolic Blood Pressure $\ge 90$ 
\item
Item~13: 
Diastolic Blood Pressure $\ge 100$ 
\item
Item~14: 
Diastolic Blood Pressure $\ge 160$ 
\end{itemize}
\item
$\{15,16,17\}$ : Systolic Blood Pressure 
\begin{itemize}
\item
Item~15: 
Systolic Blood Pressure $\ge 130$ 
\item
Item~16: 
Systolic Blood Pressure $\ge 140$ 
\item
Item~17: 
Systolic Blood Pressure $\ge 160$ 
\end{itemize}
\item
$\{18,19,20,21,22,23,24\}$ : Urine Protein 
\begin{itemize}
\item
Item~18: 
Urine Protein $= 1$ 
\item
Item~19: 
Urine Protein $= 2$ 
\item
Item~20: 
Urine Protein $= 3$ 
\item
Item~21: 
Urine Protein $\ge 1$ 
\item
Item~22: 
Urine Protein $\ge 2$ 
\item
Item~23: 
Urine Protein $\ge 3$ 
\item
Item~24: 
Urine Protein $\ge 4$ 
\end{itemize}
\item
$\{25,26,27,28,29,30,31,32,33,34,35\}$ : eGFR 
\begin{itemize}
\item
Item~25: 
eGFR $< 30$ 
\item
Item~26: 
eGFR $< 45$ 
\item
Item~27: 
eGFR $< 50$ 
\item
Item~28: 
eGFR $< 60$ 
\item
Item~29: 
eGFR $< 90$ 
\item
Item~30: 
eGFR $\ge 30$ 
\item
Item~31: 
$30 \le$ eGFR $< 45$
\item
Item~32: 
$30 \le$ eGFR $< 60$
\item
Item~33: 
$30 \le$ eGFR $< 90$
\item
Item~34: 
$45 \le$ eGFR $< 60$
\item
Item~35: 
$60 \le$ eGFR $< 90$
\end{itemize}
\item
$\{36,37,38,39,40\}$ : Diabetes (Visit history)
\begin{itemize}
\item
Item~36: 
Visit (this year)
\item
Item~37: 
Treatment (ongoing) 
\item
Item~38: 
Visit (previous year) 
\item
Item~39: 
Visit (between year before previous year and the previous year)
\item
Item~40: 
No visit (by two months after health checkup)
\end{itemize}
\item
$\{41,42,43\}$ : Hypertension (Visit history)
\begin{itemize}
\item
Item~41: 
Visit (this year)
\item
Item~42: 
Visit (previous year)
\item
Item~43: 
Visit (by three months after health checkup)
\end{itemize}
\item
$\{44\}$ : Diabetes (Medication)
\item
$\{45\}$ : Visit to medical institution
\item
$\{46,47\}$ : Diabetes (Medical history)
\begin{itemize}
\item
Item~46: 
There is a medical history
\item
Item~47: 
Treatment interruption
\end{itemize}
\item
$\{48\}$ : Hypertension (Treatment interruption)
\end{itemize}

\subsection{Data generation} 

The term ``sd'' means the standard deviation. 
If we write 
$\{0\colon x, 1\colon y\}$ for some item, then 
this means that 
the value of this item is $0$ with probability $x$ and 
$1$ with probability $y$.

\begin{itemize}
\item
Health checkup history: $\{0:0.449, 1:0.551\}$\\
This is created with reference to the following data. 
\begin{itemize}
\item
\url{https://www.kokuho.or.jp/statistics/tokutei/sokuhou/}
\end{itemize}
\item
Fasting Blood Glucose: 
$\min = 20$, $\max = 600$, ${\rm mean} = 97.78$, ${\rm sd} = 21.8$\\
This is created with reference to the following data. 
\begin{itemize}
\item
\url{https://www.kyoukaikenpo.or.jp/g7/cat740/sb7240/h28houkokusho/}
\end{itemize}
\item
Casual Blood Glucose:
$\{120:0.889, 130:0.039, 170:0.052, 210:0.020\}$\\
This is created with reference to the following data. 
\begin{itemize}
\item
\url{https://www.mhlw.go.jp/stf/seisakunitsuite/bunya/0000177221_00014.html}
\end{itemize}
The original data consists of the following four categories 
$x < 126$, $126 \le x < 140$, $140 \le x < 200$, and $200 \le x$. 
In our data, we transform these categories into 
$x = 120$, $x = 130$, $x = 170$, and $x = 210$. 
\item
HbA1c: 
$\min = 3$, $\max = 20$, ${\rm mean} = 5.19$, ${\rm sd} = 0.73$\\
This is created with reference to the following data. 
\begin{itemize}
\item
\url{https://www.kyoukaikenpo.or.jp/g7/cat740/sb7240/h28houkokusho/}
\end{itemize}
\item
Diastolic Blood Pressure:
$\min = 30$, $\max = 150$, ${\rm mean} = 75.45$, ${\rm sd} = 12.15$\\
This is created with reference to the following data. 
\begin{itemize}
\item
\url{https://www.kyoukaikenpo.or.jp/g7/cat740/sb7240/h28houkokusho/}
\end{itemize}
\item
Systolic Blood Pressure: 
$\min = 60$, $\max = 300$, ${\rm mean} = 120.63$, ${\rm sd} = 17.11$\\
This is created with reference to the following data. 
\begin{itemize}
\item
\url{https://www.kyoukaikenpo.or.jp/g7/cat740/sb7240/h28houkokusho/}
\end{itemize}
\item
Urine Protein: 
$\{1:0.853, 2:0.100, 3:0.034, 4:0.010, 5:0.003\}$\\
This is created with reference to the following data. 
\begin{itemize}
\item
\url{https://www.mhlw.go.jp/stf/seisakunitsuite/bunya/0000177221_00010.html}
\end{itemize}
The original data consists of the following five categories 
$-$, $\pm$, $+$, $++$, and $+++$. 
In our data, we transform these categories into 
$1$, $2$, $3$, $4$, and $5$. 
\item
eGFR: $\min = 1$, $\max = 500$, ${\rm mean} = 79.56$, ${\rm sd} = 14.54$\\
This is created with reference to the following data. 
\begin{itemize}
\item
\url{https://www.kyoukaikenpo.or.jp/g7/cat740/sb7240/h28houkokusho/}
\end{itemize}
\item
Diabetes (There is a medical history): $\{0:0.830, 1:0.170\}$\\
This is created with reference to the following data. 
\begin{itemize}
\item
\url{https://www.e-stat.go.jp/dbview?sid=0003224455}
\end{itemize}
\item
Diabetes (Treatment (ongoing)): 
$\{0:0.888, 1:0.112\}$\\
This is created with reference to the following data. 
\begin{itemize}
\item
\url{https://www.e-stat.go.jp/dbview?sid=0003224454}
\item
\url{https://www.e-stat.go.jp/dbview?sid=0003224455}
\end{itemize}
The above probability distribution was calculated as follows. 
The ratio of people who was pointed out as diabetes is $0.17$, and 
the ratio of people who had a history of treatment among people with diabetes is $0.657$. 
Thus, we have $0.17 \times 0.657 = 0.112$.
\item
Diabetes (Visit (this year)): 
$\{0:0.888, 1:0.112\}$\\
This is created with reference to the following data. 
\begin{itemize}
\item
\url{https://www.e-stat.go.jp/dbview?sid=0003224454}
\item
\url{https://www.e-stat.go.jp/dbview?sid=0003224455}
\end{itemize}
The above probability distribution was calculated in the
same way as Diabetes (Treatment (ongoing)).
\item
Diabetes (Visit (previous year)): 
$\{0:0.888, 1:0.112\}$\\
This is created with reference to the following data. 
\begin{itemize}
\item
\url{https://www.e-stat.go.jp/dbview?sid=0003224454}
\item
\url{https://www.e-stat.go.jp/dbview?sid=0003224455}
\end{itemize}
The above probability distribution was calculated in the
same way as Diabetes (Treatment (ongoing)).
\item
Diabetes (Visit (between year before previous year and the previous year)): 
$\{0:0.888, 1:0.112\}$\\
This is created with reference to the following data. 
\begin{itemize}
\item
\url{https://www.e-stat.go.jp/dbview?sid=0003224454}
\item
\url{https://www.e-stat.go.jp/dbview?sid=0003224455}
\end{itemize}
The above probability distribution was calculated in the
same way as Diabetes (Treatment (ongoing)).
\item
Diabetes (No visit (by two months after health checkup)): 
$\{0:0.112, 1:0.888\}$\\
This is created with reference to the following data. 
\begin{itemize}
\item
\url{https://www.e-stat.go.jp/dbview?sid=0003224454}
\item
\url{https://www.e-stat.go.jp/dbview?sid=0003224455}
\end{itemize}
The above probability distribution was calculated in the
same way as Diabetes (Treatment (ongoing)).
\item
Diabetes (Treatment interruption): 
$\{0:0.942, 1:0.058\}$\\
This is created with reference to the following data. 
\begin{itemize}
\item
\url{https://www.e-stat.go.jp/dbview?sid=0003224454}
\item
\url{https://www.e-stat.go.jp/dbview?sid=0003224455}
\end{itemize}
The above probability distribution was calculated as follows.
The ratio of people who was pointed out as Diabetes is $0.17$, and 
the ratio of people who had no history of treatment among people with diabetes is $0.343$. 
Thus, we have $0.17 \times 0.343 = 0.058$.
\item
Diabetes (Medication): 
$\{0:0.919, 1:0.081\}$\\
This is created with reference to the following data. 
\begin{itemize}
\item
\url{https://www.e-stat.go.jp/dbview?sid=0003224185}
\end{itemize}
\item
Hypertension (Visit (this year)): 
$\{0:0.868, 1:0.132\}$\\
This is created with reference to the following data. 
\begin{itemize}
\item
\url{https://www.e-stat.go.jp/dbview?sid=0004002598}
\end{itemize}
The original data consists of the numbers of essential hypertension
for every ten ages. 
We added these numbers from 45 to 74, and divided this by the total number of 
people from 45 to 74.
\item
Hypertension (Visit (previous year)): 
$\{0:0.868, 1:0.132\}$\\
This is created with reference to the following data. 
\begin{itemize}
\item
\url{https://www.e-stat.go.jp/dbview?sid=0004002598}
\end{itemize}
The above probability distribution was calculated in the
same way as Hypertension (Visit (this year)).
\item
Hypertension (Visit (by three months after health checkup)) : 
$\{0:0.868, 1:0.132\}$\\
This is created with reference to the following data. 
\begin{itemize}
\item
\url{https://www.e-stat.go.jp/dbview?sid=0004002598}
\end{itemize}
The above probability distribution was calculated in the
same way as Hypertension (Visit (this year)).
\item
Hypertension (Treatment interruption): 
$\{0:0.836, 1:0.164\}$\\
This is created with reference to the following data. 
\begin{itemize}
\item
\url{https://www.e-stat.go.jp/dbview?sid=0003224458}
\end{itemize}
The original data contain the ratio of people with medicine (0.657) among 
people with hypertension (0.477).
Thus, we calculated the ratio of people without medicine as 
$0.477 \times 0.343 = 0.164$.
\item
Visit to medical institution: 
$\{0:0.706, 1:0.294\}$\\
This is created with reference to the following data. 
\begin{itemize}
\item
\url{https://www.mhlw.go.jp/stf/seisakunitsuite/bunya/0000177221_00014.html}
\end{itemize}
The original data consists of the numbers of outpatients
for every five ages. 
We added these numbers from 45 to 74, and divided this by the total number of 
people from 45 to 74.
\end{itemize}

\end{document}